\documentclass[pdflatex,sn-mathphys]{sn-jnl}% Math and Physical Sciences Reference Style
\jyear{2021}%

\theoremstyle{thmstyleone}%
\newtheorem{theorem}{Theorem}%  meant for continuous numbers
\theoremstyle{thmstyletwo}%
\newtheorem{example}{Example}%

\theoremstyle{thmstylethree}%

\usepackage[leqno]{mathtools}
\usepackage{enumerate}
\usepackage[mathscr]{eucal}
\normalbaroutside
\usepackage{tikz}
\usepackage[center]{caption}

\usepackage{amsmath}
\usepackage{amssymb}
\usepackage{amsthm}

\usepackage{mathtools}
\usepackage{stmaryrd}
\usepackage{tikz-cd}

\theoremstyle{plain}

\newtheorem{cor}{Corollary}[theorem]
\newtheorem{lemma}{Lemma}[section]

\theoremstyle{definition}

\begin{document}

\title[]{Subgroup perfect codes of $ S_n $ in Cayley graphs}


\author[1]{\fnm{Ankan} \sur{Shaw}}\email{ankanf22@gmail.com}
\equalcont{These authors contributed equally to this work.}

\author[1]{\fnm{Shibesh} \sur{Kotal}}\email{shibeshkotal656@gmail.com}

\author*[1]{\fnm{Satya} \sur{Bagchi}}\email{sbagchi.maths@nitdgp.ac.in}


\equalcont{These authors contributed equally to this work.}

\equalcont{These authors contributed equally to this work.}

\affil[1]{\orgdiv{Department of Mathematics}, \orgname{National Institute of Technology Durgapur}, \orgaddress{\city{Durgapur}, \postcode{713209}, \state{West Bengal}, \country{INDIA}}}

%\affil[2]{\orgdiv{Department of Mathematics}, \orgname{National Institute of Technology Durgapur}, \orgaddress{\street{Street}, \city{Durgapur}, \postcode{713209}, \state{West Bengal}, \country{INDIA}}}

%\affil[3]{\orgdiv{Department}, \orgname{Organization}, \orgaddress{\street{Street}, \city{City}, \postcode{610101}, \state{State}, \country{Country}}}


\abstract{A perfect code in a graph $\Gamma = (V, E)$ is a subset $C$ of $V$ such that no two vertices in $C$ are adjacent and every vertex in $V \setminus C$ is adjacent to exactly one vertex in $C$. A subgroup $H$ of a group $G$ is called a subgroup perfect code of $G$ if there exists a Cayley graph of $G$ which admits $H$ as a perfect code. In this work, we present a classification of cyclic 2-subgroup perfect codes in $ S_n$. We analyze these subgroup codes, detailing their structure and properties. We extend our discussion to various classes of subgroup codes in the symmetric group $ S_n $, encompassing both commutative and non-commutative cases. We provide numerous examples to illustrate and support our findings.}

%%================================%%
%% Sample for structured abstract %%
%%================================%%

%\subjclass[2020]{Primary 94B05, 94B15, 20C05}


\maketitle

\section{Introduction}

Every group and every graph taken into consideration in the study are finite. In this work, we use conventional group-theoretic vocabulary and notation, which are available, for instance, in \cite{Huang2018}.\\

\noindent Consider a graph  $\Gamma$ with edges set $E(\Gamma)$ and vertices set  $V(\Gamma)$. If there is no path in $\Gamma$ connecting two vertices, the distance in $\Gamma$ between them is equal to $ \infty$. Otherwise, it equals the shortest path length between the two vertices. Let $ t $ be a natural number and $ C $ be a subset of $V(\Gamma)$. If every vertex of $ \Gamma $ is at a distance no greater than $t$ from precisely one vertex in $C$, then $ C $ \cite{Chen2020,Krotov2020} is said to be a perfect $t$-code in $ \Gamma $. A perfect code is generally a perfect $1$-code. In equivalent terms, if $C$ is an independent set of $\Gamma$ and each vertex in $V(\Gamma) \setminus C$ is adjacent to precisely one vertex in $C$, then subset $C$ of $V(\Gamma)$ is a perfect code in $\Gamma$. A perfect code in a graph is also referred to as an efficient dominating set \cite{Dejter2003, Lee2001} of the graph.\\

\noindent Initially, Biggs \cite{Biggs1973}  made known the concept of perfect $t$-codes in a graph. He defined perfect $t$-codes under the Hamming and Lee metrics as a generalization. In \cite{MacWilliams1977}, the authors showed that the Hamming distance of two code words of length $n$ over an alphabet of size $ m > 1 $ corresponds precisely to the graph distance in the Hamming graph $ H(n, m) $. Thus, under the Hamming metric, perfect $t$-codes in $H(n, m)$ are precisely those defined in the classical context. Similarly to this, the graph distance in the Cartesian product $L(n, m)$ of $n$ copies of the cycle of length $m$ equals the Lee distance \cite{Horak} between words of length $n$ across an alphabet of size $m > 2$.\\

\noindent In coding theory, the question of whether perfect codes exist under a particular metric is crucial. In the early $ 1970 $s, over Galois fields, the complete set of parameters for which perfect codes exist was categorized, see \cite{Tietäväinen, van1971, Zinoviev1973}. A conjecture made in $1970$ by Golomb and Welch \cite{Golomb1970} stated that there does not exist any $ n $ length $q$-ary perfect $r$-codes with respect to the Lee metric for all $n\geq 3$, $r \geq 2$, and $ q > 2r$. Despite extensive study spanning more than half a century, this conjecture remains unresolved \cite{Horak}. Among the most effective methods available for demonstrating nonexistence findings for perfect codes is the well-known Lloyd's Theorem \cite{Lloyd}. In distance-transitive graphs, Biggs \cite{Biggs1973} extended Lloyd's Theorem to perfect codes. A lot of research has been focused on perfect codes in distance-transitive graphs and association schemes \cite{Brouwer1989} since the groundbreaking studies of Biggs \cite{Biggs1973}. In \cite{Chihara}, Chihara showed that there exist infinitely many classes of classical distance-regular graphs that do not have any non-trivial perfect codes. She also made the same statement for the Grassmann graphs and the bilinear forms graphs. In \cite{Martin}, The authors proved that no perfect codes exist in such graphs with a new approach. One significant class of distance-regular Cayley graphs is the Doob graph family. The authors characterized all possible subgroup perfect code parameters in Doob graphs in \cite{Shi}. In \cite{Krotov2020}, the author gave a necessary and sufficient condition for the existence of the perfect codes in the Doob graph. For more treatment on perfect codes in distance-regular graphs, see \cite{Thas1980} and \cite{Hammond}.
 
\noindent In recent years, perfect codes in Cayley graphs have drawn a lot of interest; see \cite{Huang2018, Dejter2003, Feng2017, Zhou2019, van1975}. In \cite{Huang2018}, Huang et al. gave a new idea about the perfect code in the Cayley graph. They defined that a subset $S$ of a group $G$ is said to be a perfect code of $G$ if it is a perfect code of some Cayley graph of $G$. In particular, if $ S $ is a subgroup of $ G $, then $ S $ is said to be a perfect subgroup of $ G $. In \cite{Huang2018}, some intriguing findings on normal subgroups of a group being perfect codes were discovered by Huang et al. These findings are then expanded to encompass general subgroups within a group \cite{Chen2020, Ma, Zhang2021}. Let $ G $ be a group and $H$ be a subgroup of $G$. In \cite{Zhang2021, Zhang2022}, the authors showed that $H$ is a perfect code of $G$ if there exists a Sylow $2$-subgroup of $H$, which is a perfect code of $G$. They also discussed the metabelian groups. In \cite{Zhang2023}, Zhou proved that $H$ is a perfect code of $G$ if and only if a Sylow $2$-subgroup of $H$ which is a perfect code of $G$ if and only if every Sylow $2$-subgroup of $H$ which is a perfect code of $G$. Consequently, characterizing a $ H $ subgroup code in $ G $ is equivalent to characterizing a $ 2 $-Sylow subgroup of $ H $ in $ G $. When dealing with higher-order groups (especially for the symmetric groups), it is challenging to decide whether an $2$-Sylow subgroup is perfect. In this paper, we focus on some subgroup codes of the symmetric group $ S_n$. We discuss whether some subgroups are perfect in $S_n$ by looking into the generators of the subgroups of $S_n$ and providing some conditions. We also construct the higher-order subgroup codes from smaller-order subgroup codes and discuss the perfectness of such codes in $S_n$.\\  

\noindent Excluding the introductory part, the present work is demarcated into three sections. In Section $2$, we give some known results and alphabets that we use in this work. We provide a complete classification of subgroup codes whose $ 2 $-Sylow subgroups are cyclic in Section $ 3 $. In Section $4$, we discuss some classes of subgroup (commutative and non-commutative) codes in $S_n$ and furnish several examples and notes concerning our results. We finish with conclusions and directions for possible future research.

\section{Preliminaries}
In this Section, we introduce some definitions and preliminary results for the reader's convenience. For a finite set $ S $, $ |S| $ denotes the number of elements in $ S $. Let $ G $ be a finite group and $ H $ a subgroup of $ G $. We use $ |G : H| $ to represent the index $ |G|/|H| $ of $ H $ in $ G $, and $ N_G(H) $ to denote the normalizer $ \{g \in G : g^{-1}Hg = H\} $ of $ H $ in $ G $. For any subset $ A $ of $ G $, $ \langle A \rangle $ denotes the subgroup of $ G $ generated by $ A $. The order of an element $ g \in G $ is denoted by $ o(g) $ and the identity element of the group $ G $ is represented by $ 1_G $. If $ G $ contains an element $ g $ such that $ o(g) = 2$, then $ g $ is called an involution of $ G $. A subgroup $ P $ of $ G $ with order $ p^a $ is said to be a $ p $-subgroup of $ G $. If, in addition, $ p $ does not divide $ |G : P| $, then $ P $ is said to be a Sylow $ p $-subgroup of $ G $.

\begin{lemma}\cite{Zhang2021}\label{lmn 2.2}
Let $G$ be a group and $H$ a subgroup of $G$. Then $H$ is a perfect code of $G$ if and only if it has a Cayley transversal in $G$.    
\end{lemma}
 
\begin{lemma}\cite{Zhang2021}
Let $ G $ be a group and $H$ a subgroup of G. Then $H$ is a perfect code of $G$ if and only if it is a perfect code of any subgroup of $G$ which contains $H$.   
\end{lemma}

\begin{lemma}\cite{Zhang2021}
Let $G$ be a group and $H$ a subgroup of $G$. If $H$ is a perfect code of $G$, then for any $ g\in G $, $ g^{-1}Hg $ is a perfect code of $G$. More specifically, if $H$ is a perfect code in $ Cay(G, S)$ for some connection set S of G, then $ g^{-1}Hg $ is a code in $ Cay(G, g^{-1}SG)$ is a perfect code in $G$. 
\end{lemma}

\begin{lemma}\cite{Huang2018}
  Let $G$ be a group and $H$ a normal subgroup of $G$. Then $H$ is a perfect code of $G$ if and only if for all $ x\in G $, $ x^2\in H $ implies $ (xh)^2 = 1 $ for some $ h\in H .$ 
\end{lemma}

\begin{lemma}\cite{Zhang2022}\label{lmn2.5}
Let $G$ be a group and $H$ a subgroup of $G$. Then $H$ is not a perfect code of $G$ if and only if there exists a double coset $ D = HxH $ with $ D = D^{-1} $ having an odd number of left cosets of $ H $ in $ G $ and containing no involution. In particular, if $H$ is not a perfect code of $G$, then there exists a $2$-element $ x\in G\setminus H $ such that $ x^2\in H $, $ |H : H\cap xHx^{-1}| $ is odd, and $ HxH $ contains no involution.
\end{lemma}

\begin{lemma}\label{lmn1.5}\cite{Zhang2022}
Let $G$ be a group and $Q$ a subgroup of $G$. Suppose that either $Q$ is a $2$-group or at least one of $|Q|$ and $|G : Q|$ is odd. Then $Q$ is a perfect code of $G$ if and only if $Q$ is a perfect code of $ N_G(Q) $. 
\end{lemma}

\begin{lemma}\cite{Zhang2022}
Let $G$ be a group and $Q$ a subgroup of $G$. Suppose that either $Q$ is a $2$-group or at least one of $|Q|$ and $|G : Q|$ is odd. Then $Q$ is a perfect code of $G$ if and only if for any $ x\in N_G(Q) $, $ x^2\in Q $ implies $(xb)^2 = 1 $ for some $ b\in Q $.
\end{lemma}

\begin{lemma}\cite{Zhang2021}
   Let $G$ be a group and $H$ a subgroup of $G$. If either the order of $H$ is odd or the index of $H$ in $G$ is odd, then $H$ is a perfect code of $G$.  
\end{lemma}

\begin{lemma}\cite{Zhang2021}
   Let $G$ be a group and $H$ a subgroup of $G$. If there exists a Sylow $2$-subgroup of $H$, which is a perfect code of $G$, then $H$ is a perfect code of $G$. 
\end{lemma}

\begin{lemma}\label{lm1}\cite{Zhang2023}
  Let $G$ be a finite group and $H$ a subgroup of $G$. Then the following statements are equivalent:
\begin{itemize}
    \item every Sylow $2$-subgroup of $H$ is a perfect code of $G$;
    \item $H$ has a Sylow $2$-subgroup which is a perfect code of $G$;
    \item $H$ is a perfect code of $G$. 
\end{itemize}
\end{lemma}

\begin{lemma}\cite{Zhang2023}
  Let $G$ be a finite group and $H$ a subgroup of $G$. Let $Q$ be a Sylow $2$-subgroup of $H$ and $P$ a Sylow $2$-subgroup of $ N_G(Q) $. Then $H$ is a perfect code of $G$ if and only if $Q$ is a perfect code of $P$.  
\end{lemma}

\noindent It follows from the Lemma \ref{lm1} that any subgroup of odd order is perfect subgroup code in $ S_n.$ Therefore, we only focus on $2-$subgroups in $S_n$.

\section{Cyclic subgroup codes in $ S_n $ }
In this section, we provide a comprehensive classification of cyclic perfect codes in the symmetric group 
$ S_n $. By examining the generator of a cyclic subgroup, we can determine if the subgroup forms a perfect subgroup code or not. We present some constructions for such codes. We begin with an obvious result. 
\begin{theorem}
Suppose $ H = \langle (x_1x_2)(x_3x_4)\cdots(x_{m-1}x_m) \rangle$ is a subgroup of $A_n$. If all the transpositions are disjoint, then $ H $ is not perfect in $S_n$.
\end{theorem}
\begin{proof}
We have $ H = \{(1), (x_1x_2)(x_3x_4)\cdots(x_{m-1}x_m)\}.$ Let $ t = (x_1x_4x_2x_3)(x_5x_8x_6x_7)\cdots (x_{m-3}x_m x_{m-2}x_{m-1}).$ Now, $tH = \{t, t^{-1}\} = (tH)^{-1}.$ It is clear that $tH$ does not contain an involution. Hence, $ H $ is not perfect code in $S_n.$    
\end{proof}

\begin{theorem}
Suppose $ H = \langle(x_1x_2)(x_3x_4)\cdots(x_{m-1}x_m)\rangle$ is a subgroup of $S_n$ such that $\dfrac{m}{2}$ is odd. If all the transpositions are disjoint, then $ H $ is perfect in $S_n$.
\end{theorem}
\begin{proof}
Let us define a set $ S = \{ x\in S_n\setminus A_n : x\notin H \}$. It follows from the definition of $H$ that $ xH\cap yH =  \emptyset$ for $ x\neq y \in S $. Since $|H| = 2 $ then $|S| = \frac{n!}{2}-1.$ Also $ S = S^{-1}.$ Consequently, $ S_n = {\cup_{x\in S}} xH\cup H .$ Hence, $S$ is Cayley transversal of $H$ in $S_n$. It follows from Lemma \ref{lmn 2.2} that $ H $ is perfect in $S_n$.
\end{proof}
\noindent To establish our forthcoming outcome, we'll employ the subsequent lemma.
\begin{lemma}\label{lmn 3.1}
Let $ G^2 = \{ g^2 : g\in S_n \}$. Then $G^2\subset A_n$ for $n\geq 2.$
\end{lemma}
\begin{proof}
It is clear that the square of a $k$-cycle is again a $k$-cycle when $k$ is odd and is the product of two $k/2$ cycles when $k$ is even. It follows that an element $x\in S_n$ is a perfect square if and only if, for every even value of $k$, the cycle structure for $x$ has an even number of $k$-cycles.
\end{proof} 

\begin{theorem}\label{thm 3}
Let $ H = \langle x \rangle $ be a cyclic group of order $ 2^m $ of $ S_n $  such that $ x\in S_n\setminus A_n $, $ m\in\mathbb{N}$. Then $ H $ is perfect in $ S_n $.    
\end{theorem}
\begin{proof}
% $ H = \{x^i: (i,2^m)=1\}\bigsqcup \{x^i: (i,2^m) \neq 1\} = H\bigsqcup xH $, where $ H = \langle x^2 \rangle $.
%Without loss of generality, let $ x = (1234\cdots2^m)t_1t_2\cdots t_s $, where $ t_1, t_2,\cdots,t_s $ are disjoint nontrivial cycles in $ S_n $ such that $x\in S_n\setminus A_n$. 
Suppose $ H $ is not a perfect code in $ S_n $. Then, by Lemma \ref{lmn2.5}, there exist $ x_1, x_2,\dots, x_l \in S_n\setminus H $ with odd $ l $ such that $ (x_1H\cup x_2H\cup\cdots\cup x_l H )^{-1} = ( x_1H\cup X_2H\cup\cdots\cup x_lH )$ and their union contains no involution. To proceed further, we consider the following two cases.\\

\noindent \textbf{Case I: } Suppose $ l = 1 $. Then $ (x_1H)^{-1}=x_1H $ and this contains no involution. Thus, $ \{ x_1^{-1}, x^{-1}x_1^{-1},\dots, x^{-2^m-1}x_1^{-1}\} = \{ x_1, x_1x,\dots, x_1x^{2^m-1}\}$. Since $ x_1H $ does not contain an involution, therefore, $ x_1^{-1} \neq x_1 $. 

\noindent Suppose $ x_1^{-1} = x_1x^{i}$ for some $ i\in N_1 $, where $ N_1 = \{ 1\leq i\leq 2^m-1 : (i, 2^m) = 1 \}$. Then $ x_1^{-2} = x^i $, which is also a generator of $ H $ as $ i\in N_1 $. Consequently, there exists $ y\in S_n $ such that $ y^2 = x $. It follows from the Lemma \ref{lmn 3.1} that $ x\in A_n$, which contradicts the given condition.\\

\noindent Suppose $ x_1^{-1} = x_1x^{i}$ for some $ i\in N_2 $, where $ N_2 = \{ 1\leq i\leq 2^m-1 : (i, 2^m)\neq 1 \}$. Consequently, $ x_1^2 = x^{-i}$. Now, we divides this case into two parts.

\noindent\textbf{Subcase IA:} When $ x_1 = x^{(-\frac{i}{2})}.$ Then $ x_1H $ contains $ x_1x^{\frac{i}{2}}.$ Therefore, $ x_1\in H $, which is a contradiction. 

\noindent\textbf{Subcase IB:} When $ x_1 = x^{(-\frac{i}{2})}p_1p_2\cdots p_t $, where $p_1, p_2,\dots, p_t $ are disjoint transpositions in $ S_n $ and $ xp_i=p_ix$ for $ 1\leq i \leq t$. Thus, $ x_1H $ contains an involution $ x_1x^{\frac{i}{2}}(=p_1p_2\cdots p_t) $. Hence, we arrive at a contradiction. \\

\noindent \textbf{Case II:}  Suppose $ l > 1 $. Then $ \{ x_1\langle x \rangle\cup x_2\langle x \rangle \cup \cdots\cup\ x_l\langle x \rangle \}^{-1} = \{ x_1\langle x \rangle\cup x_2\langle x \rangle \cup \cdots\cup\ x_l\langle x \rangle \} $. Since $ H $ is cyclic, therefore, $(x_1\langle x \rangle\ )^{-1} = x_i\langle x \rangle $ for some $ i\in \{ 1, 2, \dots, l \} $. Then there exists $ i\in\{ 1, 2,\dots, l \} $ such that $ (x_i\langle x \rangle)^{-1} = \ x_i\langle x \rangle $ as $ l $ is odd. Thus, we arrive at a situation like Case I. This completes the proof.
\end{proof}

\begin{theorem}\label{Cyc2}
Let $ H = \langle x \rangle $ be a cyclic subgroup of order $ 2^m $ such that $ x\in A_n $, $m\in\mathbb{N}$. If $ x $ has a product of an odd number of same type of disjoint cycles, then $ H $ is perfect in $ S_n $; otherwise, $ H $ is not perfect in $ S_n $. 
\end{theorem}

\begin{proof}
 We have $ H = \langle x \rangle $, where $ x $ has a product of an odd number of same type disjoint cycles and order of $ x $ is $2^m$. Then it follows from the Lemma \ref{lmn 3.1} that there does not exist $y\in S_n $ such that 
 \begin{equation}\label{eqn777777777777777777777777777777}
      y^2 = x.\\ 
 \end{equation}
 Suppose $ H $ is not perfect in $ S_n $. Then by Lemma \ref{lmn2.5}, there exist $ x_1, x_2,\dots, x_l \in S_n\setminus H $ such that $ K:= (x_1H\cup x_2H\cup\cdots\cup x_lH) = (x_1H\cup x_2H\cup\cdots\cup x_lH)^{-1}  $ with odd $ l $ and $ K $ contains no involution. We consider the following two cases similar to Theorem \ref{thm 3}.\\

\noindent \textbf{Case I: } Suppose $ l = 1 $. Then $ (x_1H)^{-1}=x_1H $ and this contains no involution. Thus, $ \{ x_1^{-1}, x^{-1}x_1^{-1},\dots, x^{-2^m-1}x_1^{-1}\} = \{ x_1, x_1x,\dots, x_1x^{2^m-1}\}$. Since $ x_1H $ does not contain an involution, therefore, $ x_1^{-1} \neq x_1 $. 

\noindent Suppose $ x_1^{-1} = x_1x^{i}$ for some $ i\in N_1 $, where $ N_1 = \{ 1\leq i\leq 2^m-1 : (i, 2^m) = 1 \}$. Then $ x_1^{-2} = x^i $, which is also a generator of $ H $ as $ i\in N_1 $. Consequently, there exists $ z\in S_n $ such that $ z^2 = x $, which is a contradiction due to equation (\ref{eqn777777777777777777777777777777}).\\

\noindent Suppose $ x_1^{-1} = x_1x^{i}$ for some $ i\in N_2 $, where $ N_2 = \{ 1\leq i\leq 2^m-1 : (i, 2^m)\neq 1 \}$. Consequently, $ x_1^2 = x^{-i}$. Now, we divide this case into two parts.

\noindent\textbf{Subcase IA:} When $ x_1 = x^{(-\frac{i}{2})}.$ Then $ x_1H $ contains $ x_1x^{\frac{i}{2}}.$ Therefore, $ x_1\in H $, which is a contradiction. 

\noindent\textbf{Subcase IB:} When $ x_1 = x^{(-\frac{i}{2})}p_1p_2\cdots p_t $, where $p_1, p_2,\dots, p_t $ are disjoint transpositions in $ S_n $ and $ xp_i=p_ix$ for $ 1\leq i \leq t$. Thus, $ x_1H $ contains an involution $ x_1x^{\frac{i}{2}}(=p_1p_2\cdots p_t) $. Hence, we arrive at a contradiction. \\

\noindent \textbf{Case II:}  Suppose $ l > 1 $. Then $ \{ x_1\langle x \rangle\cup x_2\langle x \rangle \cup \cdots\cup\ x_l\langle x \rangle \}^{-1} = \{ x_1\langle x \rangle\cup x_2\langle x \rangle \cup \cdots\cup\ x_l\langle x \rangle \} $. Since $ H $ is cyclic, therefore, $\langle x \rangle\ x_1^{-1} = x_i\langle x \rangle $ for some $ i\in \{ 1, 2, \dots, l \} $. As $ l $ is odd, then there exists $ i\in\{ 1, 2,\dots, l \} $ such that $ (x_i\langle x \rangle)^{-1} = \ x_i\langle x \rangle $. Thus, we arrive at a contradiction due to Case I.

\noindent Hence, $ H $ is perfect in $ S_n $.\\

\noindent Let $ H = \langle x \rangle $ such that $ x $ does not have a product of an odd number of same type disjoint cycles. We claim that there exists a nontrivial coset $ x_1H$ with $ (\ x_1H \ )^{-1} = x_1H $ and contains no involution. Since  $ x $ does not have a product of an odd number of the same type of disjoint cycles, therefore, there exists $ x_1\in S_n $ such that $ x_1^2 = x $. Now, $ (\ x_1x^i \ )^{-1} = x_1(\ x_1^{-1}x^{-i}x_1^{-1} \ ) $ for $ 1\leq i \leq 2^m$. Then $  x_1^{-1}x^{-i}x_1^{-1} = x_1^{-1}x_1^{-2i}x_1^{-1} = x_1^{-2(i+1)} = x^{-(i+1)} \in H $. Hence, $ (\ x_1x^{i} \ ) ^ {-1} \in x_1H $ for  $ 1\leq i \leq 2^m$ and consequently, $ (x_1H)^{-1} = x_1H $. Next, we claim that $ x_1H $ does not contain an involution. Suppose $ x_1H $ contains an involution, say $ x_1x^j $ for some $ 1\leq j \leq 2^m $. Then $ (x_1x^j)^{-1} = x_1x^j $. Consequently, $ x_1^{-1}x^{j+1}x_1 = x^{-j} $. This shows that $ o (x^{j+1}) = o(x^{j}) $, which is a contradiction as $ o(x) $ is even. Hence, $ H $ is not perfect in $ S_n $. 
\end{proof}

\begin{cor}
Let $ H $ be a cyclic subgroup of $ S_n $ which is not perfect in $ S_n$. Then $ H \subset A_n $.
\end{cor}

\noindent Note that if a cyclic subgroup of $ S_n $ is perfect, that subgroup may or may not be a subset of $ A_n $. This is solely dependent on the subgroup's generators. This gives a general idea to construct a cyclic perfect code in $S_n$.   
\begin{example}
Let $ G = S_{10} $, $ K = \langle (1\; 2\; 3\; 4)(5\; 6), (7\; 8\; 9) \rangle $ and $ H =  \langle (1 \;2\; 3\; 4)(5\; 6) \rangle .$ Then $ H $ is Sylow $ 2 $-subgroup of $ K.$ It follows from Theorem \ref{Cyc2} that $ H $ is not a perfect code in $ G$. Therefore, $ K $ is not perfect in $ G $ due to Lemma \ref{lm1}.  
\end{example}

%%%%%%%%%%%%%%%%%%%%%%%%%%%%%%%%%%%%%%%%%%%%%%%%%%%%%%%%%%%%%%%%%%%%%%%%%%%%%%%%%%%%%%%%%%%%

\section{Some non-cyclic subgroup codes in $ S_n $}
In this section, we explore the construction of non-cyclic subgroup codes, both commutative and some non-commutative, that are not perfect within $ S_n $. Through an example, we demonstrate that the isomorphism of the group may not preserve the perfectness of the subgroup code.

\begin{theorem}\label{thm5}
Let $ H = \langle x_1, x_2, \dots, x_m \rangle $ be a commutative subgroup of order $ 2^k$ in $S_n$ such that $ x^2 = x_1$ for some $ x\in S_n\setminus H$, $ m, k \in\mathbb{N}$. If $ xx_i = x_ix $  for $i \in \{ 2, 3, \dots, m \} $, then $H$ is not perfect in $S_n.$
\end{theorem}

\begin{proof}
Let $ H_1 = \langle x_1\rangle $ and $ H = H_1 \bigsqcup_{i=1}^{p} w_iH_1$, where $ w_i = x_2^{i_2}\cdots x_m^{i_m}$; $ i_2, \dots, i_m \in \mathbb{N}\cup \{0\}$. It follows from Theorem \ref{Cyc2} that $ (xH_1)^{-1} =  xH_1 $ and it contains no involution. Then
\begin{align*}
(xH)^{-1} & = (xH_1)^{-1} \bigsqcup_{i=1}^{p } (xw_iH_1)^{-1}\\
         & = (xH_1)\bigsqcup_{i=1}^{p } H_1w_{i}^{-1}x^{-1}.\\
\end{align*} 

\noindent Since  $ xx_i = x_ix $ for $ 2\leq i\leq m$, therefore, $w_j^{-1}x^{-1} = x^{-1} w_j^{-1} $ for $ 1\leq j \leq p $. Then $ H_1 w_i^{-1}x^{-1}  = xH_1w_i^{-1} $ as $ (xH_1)^{-1} =  xH_1 .$ Since $ H $ is commutative, therefore, $  H_1 w_i^{-1}x^{-1} = x w_i^{-1}H_1.$ Thus, $(x w_iH_1)^{-1} = x w_i^{-1}H_1$ for $ 1 \leq i \leq p.$
    To proceed further, we consider the following two cases.\\
    
\noindent \textbf{Case I:} When  $ w_i^{-1}  \neq  w_i$.
    Then $ xw_iH_1$ does not contain an involution.\\
    
    %If $ (x w_iH)^{-1} = x w_i^{-1}H $ for some $ j \in \{1, 2, \dots, 2^t-1\} .$ Then $ w_i^{-j}  =  w_i^j.$ This implies that $j = 2^{t-1}$ for $ j \neq 2^{t-1} $ and $ j \in \{1, 2, \dots, 2^t-1\}. $ Thus, $  (x w_i^jH)^{-1} = x w_i^jH .$ Therefore, for $ j \in \{1,2,\dots,2^t-1\} \setminus \{2^{t-1}\}, $ $ x w_i^iH $ does not contain an involution.\\

\noindent \textbf{Case II:} When  $ w_i^{-1} =  w_i$.
    Suppose $ x w_iH_1 $ contains an involution, say  $ x w_ih_1 ,$ for some $ h_1 \in H_1 $. Then $(x w_ih_1) (x w_ih_1)  = 1_{S_n}.$ This implies $(xh_1)^2 = 1_{S_n}.$ Consequently, $ xH_1 $ contains an involution which is a contradiction. Hence, the desired result follows.
\end{proof}

\noindent \textbf{Note:} Observe that Theorem \ref{thm5} inspires the construction of higher-order codes that are not perfect in $ S_n $ based on the given subgroup codes which are not perfect in $S_n$. Take an element $ z\in S_n \setminus H $ such that $ zx = xz $ and $ zx_i = x_iz $ for $ 2\leq i \leq m $. Then $ \langle H, z \rangle $ is not perfect in $ S_n $.

%%%%%%%%%%%%%%%%%%%%%%%%%%%%%%%%%%%%%%%%%%%%%%%%%%%%%%%%%%%%%%%%%%%%%%%%%%%%%%%%%%%%%%%%%%%%%%%%%%%%%%%%%%%%%

%%%%%%%%%%%%%%%%%%%%%%%%%%%%%%%%%%%%%%%%%%%%%%%%%%%%%%%%%%%%%%%%%%%%%%%%%%%
\subsection{Non-commutative subgroup codes}

Next, we will discuss certain non-abelian codes that are not perfect in $ S_n $. These non-abelian subgroups are isomorphic to $ D_n $, $\mathbb{Z}_n \rtimes \mathbb{Z}_2$, subgroups with a minimal generating set of cardinality $2$ with some conditions. Utilizing these subgroups, we will construct higher-order subgroup codes that are also not perfect in $ S_n $.

\begin{theorem}\label{k11}
     Let $ \langle x_1, x_2 \rangle $ be a non-commutative subgroup of order $2^{l+1} $ such that $ y^{2}=x_1^{-1} $  and $o(x_1) = 2^l $, $ o(x_2) = 2$ for some $ y\in S_n$, $ l\in \mathbb{N}$. If $x_2y^{-1} = y^{-k }x_2 $ for some odd integer $ k $ such that $ k\not\equiv -1 \pmod{2^{l+1}} $, then $ \langle x_1, x_2 \rangle $ is not perfect in $S_n$. 
\end{theorem}
\begin{proof}
Let $ H = \langle x_1, x_2 \rangle   $. We have $x_2y^{-1} = y^{-k }x_2 $. Then $ x_2x_1=(x_2y^{-1})y^{-1}=y^{-2k}x_2=(y^{-2})^kx_2=x_1^kx_2$. Since $ y^{2}=x_1^{-1} $ and $ o(x_1) = 2^l $, therefore, $ o(y) = 2^{l+1} $. Hence, $ y \notin H $ due to non-commutativity of $H$. Thus, we have 
\begin{equation}\label{Thm_equ1}
yH = y\langle x_1 \rangle \bigsqcup yx_2 \langle x_1 \rangle.
\end{equation}

\noindent Now, 
\begin{align*}
         [y \langle x_1 \rangle]^{-1} & = \{yx_1^i:i= 0, 1, \dots, 2^{l}-1 \}^{-1}\\
    & = \{ x_1^{-i}y^{-1}:i= 0, 1, \dots, 2^{l}-1 \}\\
    & = \{ y^{2i-1}:i= 0, 1, \dots, 2^{l}-1\}\\
    & = \{ yy^{2i-2}:i= 0, 1, \dots, 2^{l}-1 \}\\   
    & = \{ y(y^2)^{i-1}:i=0, 1, \dots, 2^{l}-1 \}\\ 
    & = \{ yx_1^{1-i}:i= 0, 1, \dots, 2^{l}-1 \}\\ 
    & = y \langle x_1 \rangle.
\end{align*}
    
\noindent Suppose there exists $ m \in \{0, 1, \dots, 2^{l}-1 \} $ such that\\
    \begin{align*}
        yx_1^{1-m} & = (yx_1^{1-m})^{-1}\\
       \iff yx_1^{1-m} & = x_1^{m-1}y^{-1}\\
       \iff yy^{-2(1-m)} & = y^{-2(m-1)}y^{-1}\\
       \iff y^{2m-1} & = y^{-2m+1}.\\
    \end{align*}
Therefore, $ 2m-1  = -2m+1  + t2^{l+1} ;$ $ t\in\mathbb N\cup \{0\}.$ This implies that $ 2m = 1+2(2^{l-1}t),$ which has no integer solution. Hence, $ [y \langle x_1 \rangle]^{-1} = y \langle x_1 \rangle $ and it contains no involution.
    
\noindent Since $ k $ is odd, then there exists $ m_1 \in \mathbb Z $ such that $ k+1 = 2m_1 .$ Now, 
    \begin{align*}
         [yx_2 \langle x_1 \rangle]^{-1} & = \{ yx_2x_1^i  :i= 0, 1, \dots, 2^{l}-1\}^{-1}\\
    & = \{ x_1^{-i}x_2y^{-1}:i= 0, 1, \dots, 2^{l}-1 \}\\
    & = \{ x_1^{-i}y^{-k}x_2:i= 0, 1, \dots, 2^{l}-1 \}\\
    & = \{ y^{(2i-k)}x_2 :i=0, 1, \dots, 2^{l}-1 \}\\   
     & = \{ yy^{(2i-k-1)}x_2 :i=0, 1, \dots, 2^{l}-1 \}\\   
    & = \{ yy^{-2(m_1-i)}x_2 :i=0, 1, \dots, 2^{l}-1 \}\\   
    & = \{ yx_1^{(m_1-i)}x_2 :i= 0, 1, \dots, 2^{l}-1 \}\\ 
    & = y \langle x_1 \rangle x_2\\
    & = yx_2 \langle x_1 \rangle.
    \end{align*}
    Suppose there exists $t \in \mathbb{N}\cup\{0\} $ such that
   \begin{align*}
       [yx_2x_1^{t}]^{-1}  & =  yx_2x_1^{t} \\
       \iff x_1^{-t}x_2y^{-1} & = yx_2x_1^{t}\\
        \iff y^{2t}y^{-k}x_2 & = yx_2y^{-2t}\\
        \iff y^{(2t-k)}x_2 & = yy^{(-2t)k}x_2.
    \end{align*}
    Therefore,
    \begin{align*}
        2t-k & = 1-2tk + t_12^{l+1}; t_1\in\mathbb N \cup \{0\}\\
        \iff (2t-1)(1+k) & = t_12^{l+1}.\\
    \end{align*}
    Since $ k $ is odd, therefore, $ k\equiv -1 \pmod{2^{l+1}} $ which contradicts the give condition.
    
\noindent Hence, from equation (\ref{Thm_equ1}), we get $ [yH]^{-1} = yH $ and it contains no involution. Therefore, $H$ is not perfect in $S_n$.
\end{proof}

\noindent We use the aforesaid result to show that group isomorphism may not preserve the perfectness of the subgroup code. The following example reflects the statement. 
\begin{example}
Let $ H_1=\langle (1\;4\;7\;6)(2\;8\;3\;5), (2\;5)(3\;8)(4\;6)\rangle $ and $ H_2 = \langle (1\;6)(2\;4)(3\;8)(5\;7), (1\;8\;5\;4)(2\;7\;3\;6)\rangle $ be subgroups of order $ 8 $ in $ S_8 $. Both $ H_1$ and $ H_2 $ are isomorphic to $ D_4$. Consequently, $ H_1\cong H_2 $. It follows from the proof of Theorem \ref{k11} that $ (1\;2\;6\;5\;7\;3\;4\;8)H_1 $ is inverse-closed and contains no involution. Consequently, $ H_1 $ is not perfect in $ S_8.$ Now, $ N_{S_8}(H_2) = H_2 \cup (4\;8)(6\;7)H_2\cup (2\;3)(6\;7)H_2\cup (2\;6)(3\;7)(4\;8)H_2\cup (2\;7)(3\;6)(4\;8)H_2\cup (2\;6\;3\;7)H_2\cup (2\;7\;3\;6)H_2.$ Then $ H_2 $ is perfect in $  N_{S_8}(H_2) $. It follows from Lemma \ref{lmn1.5} that $ H_2 $ is perfect in $ S_8 $. Hence, we can see that group isomorphism may not preserve the perfectness of the subgroup code. 
\end{example}

\begin{theorem}\label{thm8}
     Let $ \langle x_1, x_2 \rangle $ be a non-commutative subgroup of order $2^{l+1} $ such that $ y^{2}=x_1^{-1} $  and $o(x_1) = 2^l $, $ o(x_2) = 2$ for some $ y\in S_n$, $ l\in \mathbb{N}$. Suppose $x_2y^{-1} = y^{-k }x_2 $ for some odd integer $ k $ such that $ k\not\equiv -1 \pmod{2^{l+1}} $. If $ x_3 y = y x_3$ and $ x_2x_3 = x_3x_2 $, then $ \langle x_1, x_2, x_3 \rangle $ is not perfect in $S_n$.
\end{theorem}
\begin{proof}
 Let $ H =  \langle x_1, x_2 \rangle $, $ H_1 = \langle x_1, x_2, x_3 \rangle$ and $ [H_1:H]=t $. Then $ H_1 = H\bigsqcup_{i=1}^{t-1}  x_3H$. Since $ x_3 y = y x_3$, therefore, $ x_3 x_1 = x_1 x_3 $. It follows from Theorem \ref{k11} that $ yH$ is inverse-closed and contains no involution. Since $ x_3 x_2 = x_2 x_3 $, $ (yH)^{-1} = yH $ and $ x_3 y = y x_3$, therefore, $ [yx_3^{i}H]^{-1} = yx_3^{-i}H$. Suppose $ yx_3H $ contains an involution. To proceed further, we consider the following two cases.\\

\noindent\textbf{Case I:} When $ x_3^{-i}=x_3^{i}$. Suppose $ yx_3^i\langle x_1\rangle $ contains an involution. Then there exists $ m \in \{0, 1, \dots, 2^{l}-1 \} $ such that\\
    \begin{align*}
        yx_3^ix_1^{1-m} & = (yx_3^ix_1^{1-m})^{-1}\\
       \iff yx_1^{1-m} & = x_1^{m-1}y^{-1}\\
       \iff yy^{-2(1-m)} & = y^{-2(m-1)}y^{-1}\\
       \iff y^{2m-1} & = y^{-2m+1}.
    \end{align*}
    Therefore, $ 2m-1  = -2m+1  + t_1 2^{l+1} ;$ $ t_1\in\mathbb N\cup \{0\}.$ This implies that $ 2m = 1+2(2^{l-1}t_1),$ which has no integer solution. Hence, $ yx_3^i\langle x_1\rangle $ contains no involution. Suppose $ yx_3^ix_2\langle x_1\rangle $ contains an involution.    
    Then there exists $ s \in\mathbb Z $ such that  \begin{align*}
       [yx_3^{i}x_2x_1^{s}]^{-1}  & =  yx_3^{i}x_2x_1^{s} \\
       \iff x_1^{-s}x_2y^{-1} & = yx_2x_1^{s}\\
        \iff y^{2s}y^{-k}x_2 & = yx_2y^{-2s}\\
        \iff y^{(2s-k)}x_2 & = yy^{(-2s)k}x_2.
    \end{align*} 
      Therefore,
    \begin{align*}
        2s-k & = 1-2sk + t_2 2^{l+1}; t_2\in\mathbb N \cup \{0\}\\
        \iff (2s-1)(1+k) & = t_22^{l+1}. 
    \end{align*}
    Since $ k $ is odd, therefore, $ k\equiv -1 \pmod{2^{l+1}}$. which contradicts the given condition.\\
    
\noindent\textbf{Case II:} When $ x_3^{-i}\neq x_3^{i}$. Then $ [yx_3^{i}H]^{-1} \neq yx_3^iH $. Consequently, $ yx_3H $ contains no involution.  

\noindent Hence, $ H_1 $ is not perfect in $ S_n$.
\end{proof}

\noindent Note that the above theorem inspires us to construct a new higher-order code, which is not perfect in $ S_n$ under some restrictions. The following corollary reflects our motivation.  
\begin{cor}\label{cor9}
 Let $ \langle x_1, x_2 \rangle $ be a non-commutative subgroup of order $2^{l+1} $ such that $ y^{2}=x_1^{-1} $  and $o(x_1) = 2^l $, $ o(x_2) = 2$ for some $ y\in S_n$, $ l\in \mathbb{N}$. Suppose $x_2y^{-1} = y^{-k }x_2 $ for some odd integer $ k $ such that $ k\not\equiv -1 \pmod{2^{l+1}} $. If $ x_i y = y x_i$, $ x_2x_i = x_ix_2 $ and $ x_ix_j = x_jx_i$ for $ i,j \in\{3,4,\dots, m \}$, then $ \langle x_1, x_2,\dots x_m \rangle $ is not perfect in $S_n$. 
\end{cor}

\noindent To facilitate our understanding of the structure of non-abelian codes in $ S_n $, we first establish the following lemma.
    \begin{lemma}\label{k2}
     Suppose $ k, l \in\mathbb N$ such that $ k< 2^l $ and $ l \ge 2 .$ If $k$ is odd, then there does not exist $ i\in\mathbb N \cup\{0\}$ such that $ k^i \equiv -1 \pmod{2^{l+1}}.$ 
    \end{lemma}
    \begin{proof}
We will prove this lemma by considering two cases.\\
       
\noindent \textbf{Case I:} When $ l = 2.$ Then $ k = 1 $ or $ 3.$ Now, $ 3^n \equiv 1 \pmod 8 $  for even $ n\in\mathbb N $ and  $ 3^n \equiv 3 \pmod 8 $  for odd $ n\in\mathbb N .$\\
       
\noindent \textbf{Case II:} When $ l\geq 3.$ We first prove that for $ n\geq 3 $,  $U_{2^n} \cong \mathbb{Z}_2 \times \mathbb{Z}_{2^{n-2}}.$ Let the order of $a$ in $ U_{2^n}$ denoted by $o(a)$. We claim that $o(5) = 2^{n-2}$ in $ U_{2^n}$ for $n\geq 3 $. To prove this, we use induction. For $ n=3 $, we have   $ U_{2^3} = \{1,3,5,7\} $. Now, $ 5^2 \equiv 1 \pmod{ 2^3} $. Therefore, $ o(5) = 2^{3-2}.$ Let $o(5) = 2^{m-2} $ in $  U_{2^m}$. Then $ 5^{2^{m-2}} = t2^m +1 $ for some odd $ t \in \mathbb N.$ Now, 
     \begin{align*}
          5^{2^{m-1}} -1 = ( 5^{2^{m-2}})^2-1 & = (t2^m+1)^2-1 \\
          & = 2(t^2 2^{2m-1}+ t2^m)\\
          & = 2^{m+1}(t^22^{m-1}+t).\\
     \end{align*}
This implies that $  5^{2^{m-1}} \equiv 1 \pmod{2^{m+1}}.$ Since $ t $ is odd, therefore, $ o(5) = 2^{m-1} $ in $ U_{2^{m+1}}.$ Hence, by mathematical induction $o(5) = 2^{n-2}$ in $ U_n$ for $ n\geq 3 .$ Now, $ 2^n-1 \equiv -1 \pmod{ 2^n } $ has order $ 2 $ in $  U_{2^n} $ and  $o( U_{2^n}) =2^{n-1} $. It is clear that $5$ and $-5$ are distinct elements in $  U_{2^n} $. Since $ o(5) = 2^{n-2}$, therefore, every element in $  U_{2^n}$ can be expressed as of the form $ (-1)^a5^b $  for $ a \in \{0,1\}, b \in \{ 0, 1, \dots, 2^{n-2}-1\} $. Claim that this representation is unique.

\noindent Suppose there exist $a' \in  \{0,1\} $ and $ b' \in \{ 0, 1, \dots, 2^{n-2}-1\} $ with $ a\neq a' $ and $ b \neq b' $ such that  $(-1)^{a'} 5^{b'} =  (-1)^{a}5^b $. This implies that $(-1)^{a'-a} 5^{b'-b} = 1 $. Consequently. $a-a'=0$ and $b'-b= t'2^{n-2}$ for some $ t'\in\mathbb{Z}$ as $o(-1) = 2 $ and $ o(5) = 2^{n-2} $ in $ U_n$. Since $ b,b' \in  \{ 0, 1, \dots, 2^{n-2}-1\} $, therefore, we get $ t= 0.$ Consequently, $ b=b'.$

\noindent Let us define a homomorphism $ \phi : \mathbb{Z}_2 \times \mathbb{Z}_{2^{n-2}} \longrightarrow U_{2^n} $ by $ \phi (a,b) = (-1)^{a}5^b .$ From above, we can say that $ \phi $ is an isomorphism. Then $ (1,0) \mapsto (-1).$ Suppose there exists $ t_1 \in \mathbb N $ such that $ k^{t_1} \equiv -1 \pmod{2^{l+1}}$ and $ k < 2^l $ with odd natural number $ k$. This implies that $ (-1)^1k^{t_1} \equiv 1 \pmod{2^{l+1}}.$ Take the group $ G =  U_{2^{l+1}} $ for $ l \geq 3 $. Then $ U_{2^{l+1}} \cong \mathbb{Z}_2 \times \mathbb{Z}_{2^{l-1}}.$ Now, $ \phi(a,b) = (-1)^{a}5^b.$ Then $\phi(1,0) = -1.$
     
\noindent Now, $ k^{t_1} \equiv -1 \pmod{ 2^{l+1}} .$ Then there exists $ (a,b) \in \mathbb{Z}_2 \times \mathbb{Z}_{2^{n-1}} $ such that $\phi ^{-1}(k)   = (a,b)$. Then $ k^{t_1}  = [\phi(a,b)]^{t_1}.$ This implies that $ k^{t_1} = \phi(t_1a,t_1b).$ Consequently, $ \phi(t_1a,t_1b)  = -1 .$ Since $ \phi $ is an isomorphism, therefore, $ t_1a = 1 $ in $\mathbb{Z}_2 $ and $ t_1b = 0 $ in $\mathbb{Z}_{2^{n-1}}.$ This implies $ a= 1 $ and $ t_1 =1 $ in $\mathbb{Z}_2 $. Consequently, $ b=0 $ in $\mathbb{Z}_{2^{n-1}}.$ Thus, $k  \equiv -1\pmod{ 2^{l+1}}$, equivalently, $k = 2^{l+1} -1$ in $U_{2^{l+1}}$. Define a set $ M = \{ x \in \mathbb{N} : x < 2^l \}$. Then $ [k] = [2^{l+1} - 1]  \cap M = \emptyset $, where $[k] $ denotes the set of remainders when the result is divided by $ 2^{l+1}$. This contradicts the existence of $ t_1 $, $ k $, and it follows the desired result.
\end{proof}

\begin{theorem}\label{thm11}
Let $ \langle x_1, x_2 \rangle $ be a non-commutative subgroup of order $2^{l+1} $ such that $o(x_1) = 2^l $, $ o(x_2) = 2$, and  $ y^{2}=x_1^{-1} $ for some $ y \in S_n $, $ l\in \mathbb{N}$. Let $ x_2x_3^{-1} = x_3^{-s}x_2 $, $ x_2y^{-1} = y^{-k}x_2 $, and $ x_3y^{-1} = y^{-k}x_3 $ for some $ s, k \in \mathbb N $. Then  $ \langle x_1, x_2, x_3 \rangle / \langle x_1, x_2 \rangle $ is cyclic.
\end{theorem}

\begin{proof}
Let $ H =  \langle x_1, x_2 \rangle $, $ H_1 =  \langle x_1, x_2, x_3 \rangle $,  $ [H_1:H] = t $, and  $ o(x_3) = t_1 $. It is given that $ x_2x_3^{-1} = x_3^{-s}x_2 $, therefore, 
\begin{equation}\label{eqn1}
  x_2x_3^{-i} = x_3^{-is}x_2.   
\end{equation}
Putting $ i = t_1-1 $ in equation $(\ref{eqn1}) $, we get $ x_2x_3 = x_3^{s}x_2$. We have $ x_2y^{-1} = y^{-k}x_2 $ and $  y^{2}=x_1^{-1} $.  So 
\begin{equation}\label{eqn2}
    x_2x_1 = (x_2y^{-1})y^{-1} = y^{-2k}x_2 = x_1^kx_2
\end{equation}
and 
\begin{equation}\label{eqn7}
    x_2y^{-i} = y^{-ik}x_2.
\end{equation}
Now, $ x_3^{2}y^{-1} = x_3(x_3y^{-1}) = x_3(y^{-k}x_3) = (x_3y^{-1})(y^{-(k-1)}x_3) = y^{-k}(x_3y^{-1})y^{-(k-2)}x_3 = y^{-2k}(x_3y^{-1})y^{-(k-3)}x_3 = y^{-3k}(x_3y^{-(k-3)})x_3^{2} .$ Continuing this process, we get $ x_3^{2}y^{-1} = y^{-k\cdot k}x_3 = y^{-k^2}x_3^2.$ Thus, for $ i \in \mathbb{N} $,
\begin{equation}\label{eqn3}
 x_3^{-i}y^{-1} = x_3^{t_1-i}y^{-1} = y^{-k^{t_1-i}}x_3^{-i}.
\end{equation}
Similarly, from equation $(\ref{eqn2})$, we get 
\begin{equation}\label{eqn4}
x_2x_1^i = x_1^{ik}x_2 
\end{equation}
 and  
 \begin{equation}\label{eqn8}
  x_3^jx_1^i = x_1 ^{ik^j}x_3^j.    
 \end{equation}
 Consequently, for $ i, j \in\mathbb{N}$,   
\begin{equation}\label{eqn5}
x_1^{-i}x_3^{-j} = x_3^{-j}x_1^{-ik^j}, \; \text{i.e.,} \; \;  x_1^{2^l-i}x_3^{t_1-j} = x_3^{t_1-j}x_1^{2^l-ik^j}.
\end{equation}

\noindent Putting $ i = 2^l-1, j = t_1-1$ in equation (\ref{eqn5}), we get $ x_1x_3 = x_3x_1^{k^{t_1-1}}$. This implies that $ x_1^ix_3 = x_3x_1^{ik^{t_1-1}}$. Now, $ x_1x_3^2 =  (x_3x_1^{(k^{t_1-1}-1)})(x_3x_1^{k^{t_1-1}})=  (x_3x_1^{(k^{t_1-1}-2)})(x_1x_3)x_1^{k^{t_1-1}} = (x_3x_1^{(k^{t_1-1}-2)})(x_3x_1^{2k^{t_1-1}}).$ Continuing this process, we get $ x_1x_3^2 = x_3^2x_1^{k^{2(t_1-1)}}$. Thus, $ x_1x_3^i = x_3^ix_1^{k^{i(t_1-1)}}$. Hence, $  x_1^ix_3^j = x_3^jx_1^{ik^{j(t_1-1)}} $ for $ i,j\in \mathbb{N}.$ Since $ x_2x_3 = x_3^{-s}x_2$ and $ x_1x_3 = x_3x_1^{k^{t_1-1}}$, therefore, $x_1x_2x_3^iH = x_3^jH$ for some $ i,j \in\mathbb{N}\cup\{0\}.$ Similarly, $ x_1^ix_2x_3^jH = x_3^kH$ for some $ i, j, k \in\mathbb{N}\cup\{0\}.$ Hence $ \langle x_1, x_2, x_3 \rangle / \langle x_1, x_2 \rangle $ is a cyclic group.
\end{proof}
\begin{theorem}\label{thm10}
Let $ \langle x_1, x_2 \rangle $ be a non-commutative subgroup of order $2^{l+1} $ such that $ y^{2}=x_1^{-1} $  and $o(x_1) = 2^l $, $ o(x_2) = 2$ for some $ y \in S_n $, $ l\in \mathbb{N}$. Let $ x_2x_3^{-1} = x_3^{-s}x_2 $, $ x_2y^{-1} = y^{-k}x_2 $, and $ x_3y^{-1} = y^{-k}x_3 $ for some $ s, k \in \mathbb N $. If  $ k $ is odd and $k < 2^l $ then $ \langle x_1, x_2, x_3 \rangle $ is not perfect in $S_n$.
\end{theorem}
\begin{proof}
 Let $ H =  \langle x_1, x_2 \rangle $, $ H_1 =  \langle x_1, x_2, x_3 \rangle $,  $ [H_1:H] = t $, and  $ o(x_3) = t_1 $. It follows from Theorem \ref{thm11} that \begin{equation}\label{eqn6}
H_1 = H \bigsqcup_{i=1}^{t-1} x_3^iH.
\end{equation}
Now,
\begin{align*}
[yx_3^i \langle x_1 \rangle]^{-1} & = \{yx_3^ix_1^j:j= 0, 1,\dots, 2^{l}-1 \}^{-1}\\
& = \{ x_1^{-j}x_3^{-i}y^{-1}:j =0, 1,\dots, 2^{l}-1 \} \\
& = \{x_3^{-i} x_1^{-jk^{i}}y^{-1} : j = 0, 1,\dots, 2^{l}-1 \}\hspace{2.5cm} (\text{from} \; (\ref{eqn5}))\\    
& =  \{x_3^{-i} y^{-1}x_1^{-jk^{i}} : j = 0, 1,\dots, 2^{l}-1 \}\hspace{2.5cm} (\text{as} \; x_1y = yx_1)\\      
& = \{  y^{-k^{(t_1-i)}}x_3^{-i} x_1^{-jk^{i}}:j = 0, 1,\dots, 2^{l}-1 \} \hspace{1.78cm} (\text{from} \; (\ref{eqn3})) \\
& = \{ yy^{-(k^{(t_1-i)}+1)}x_3^{-i}x_1^{-jk^{i}} :j= 0, 1,\dots, 2^{l}-1 \},\\
\end{align*}
and these hold for $ 1\leq i \leq t-1.$ Since $k$ is odd, therefore, there exist $ m', r', n' \in  \mathbb N \cup \{0\} $ such that $ k^{t_1-i}+1 = 2m'$ and $ m'=n'2^l + r' $ with $ 0\leq r' < 2^l $. Then
\begin{align*}
[yx_3^i \langle x_1 \rangle]^{-1} & = \{ yy^{-2m'}x_3^{-i} x_1^{-jk^{i}} :j = 0, 1,\dots, 2^{l}-1 \}. \\
& = \{ yx_1^{m'}x_3^{-i} x_1^{-jk^{i}} :j = 0, 1,\dots, 2^{l}-1 \} \\
& = \{ yx_1^{r'}x_3^{-i} x_1^{-jk^{i}} :j = 0, 1,\dots, 2^{l}-1 \} \\
& = \{ yx_1^{2^l-(2^l-r')}x_3^{-i} x_1^{-jk^{i}} :j = 0, 1,\dots, 2^{l}-1 \} \\
& = \{ yx_3^{-i} x_1^{2^l-(2^l-r')k^i}x_1^{-jk^{i}} :j = 0, 1,\dots, 2^{l}-1 \} \hspace{1.95cm}(\text{from} \; (\ref{eqn5})) \\
& = \{ yx_3^{-i} x_1^{(r'-j)k^{i}} :j = 0, 1,\dots, 2^{l}-1 \}\\
& =  y x_3^{-i}\langle x_1 \rangle . \\
\end{align*}

\noindent Now,
\begin{align*}
[yx_3^ix_2 \langle x_1 \rangle]^{-1} & = \{yx_3^ix_2x_1^j:j = 0, 1,\dots, 2^{l}-1 \}^{-1}\\
    & = \{ x_1^{-j}x_2x_3^{-i}y^{-1}:j = 0, 1,\dots, 2^{l}-1 \} \\
      & = \{ x_1^{-j}x_2y^{-k^{(t_1-i)}}x_3^{-i} : j = 0, 1,\dots, 2^{l}-1 \}\hspace{2.3cm} (\text{from} \; (\ref{eqn3}))\\
       & = \{ x_2x_1^{-jk} y^{-k^{(t_1-i)}}x_3^{-i} :j = 0, 1,\dots, 2^{l}-1 \}\hspace{.5cm} (\text{with the help of} \; (\ref{eqn2}))\\
    & = \{ x_2 y^{2jk-(k^{(t_1-i)})}x_3^{-i} :j = 0, 1,\dots, 2^{l}-1 \} \hspace{1.65cm} (\text{as} \; y^2 = x_1^{-1})\\
     & = \{ x_2y^{k} y^{2jk-(k^{(t_1-i)}-k)}x_3^{-i} :j = 0, 1,\dots, 2^{l}-1 \}\\
     % & = \{ yx_2x_1^mx_3^{-i} :j = 0, 1,\dots, 2^{l}-1 \} \hspace{2cm} (\text{by the given condition}),\\
\end{align*}
and these holds for $ 1\leq i \leq t-1.$ Since $k$ is odd, therefore, there exists $ m_j \in\mathbb{Z} $ such that $ 2jk-k^{(t_1-i)}-k = -2m_j$.\\

\noindent\textbf{Case A:} When $ m_j\leq 0$.  Then  
        \begin{align*}
             [yx_3^ix_2 \langle x_1 \rangle]^{-1} & =  \{ yx_2x_1^{m_j}x_3^{-i} :j = 0, 1,\dots, 2^{l}-1 \} \hspace{1cm}(\text{by the given condition})\\
 & = \{ yx_2x_1^{-|m_j|}x_3^{-i} : j = 0, 1, \dots, 2^{l}-1 \} \\
        & = \{ yx_2x_3^{-i} x_1^{-|m_j|k^i} :j = 0, 1,\dots, 2^{l}-1 \}\hspace{2.4cm} (\text{from} \; (\ref{eqn5}))  \\
         & = \{ yx_3^{-is}x_2 x_1^{m_jk^i} :j = 0, 1,\dots, 2^{l}-1 \} \hspace{2.65cm} (\text{from} \; (\ref{eqn1}))\\
        & =  y x_3^{-is}x_2 \langle x_1 \rangle. \\             
\end{align*}
\textbf{Case B:} When $ m_j > 0$. Then there exist $ n_j, r_j\in\mathbb{N}\cup\{0\}$ such that $ m_j = n_j2^l+r_j $ with $ 0\leq r_j<2^l$. Therefore, 
 \begin{align*}
     [yx_3^ix_2 \langle x_1 \rangle]^{-1} & =  \{ yx_2x_1^{r_j}x_3^{(t_1-i)} : j = 0, 1, \dots, 2^{l}-1 \} \\   
     & = \{ yx_2x_1^{2^l-(2^l-r_j)}x_3^{(t_1-i)} : j = 0, 1, \dots, 2^{l}-1 \} \\
        & = \{ yx_2x_3^{t_1-i} x_1^{2^l-(2^l-r_j)k^i} :j = 0, 1,\dots, 2^{l}-1 \}\hspace{1.5cm} (\text{from} \; (\ref{eqn5}))  \\
         & = \{ yx_3^{-is}x_2 x_1^{r_jk^i} :j = 0, 1,\dots, 2^{l}-1 \} \hspace{2.77cm} (\text{from} \; (\ref{eqn1}))\\
        & =  y x_3^{-is}x_2 \langle x_1 \rangle. \\
        \end{align*}
The above arguments and the Theorem \ref{thm8} evidence that $yH_1$ is inverse-closed. We consider the following four cases to show that $ yH_1 $ contains no involution. \\

\noindent \textbf{Case I:} When $ x_3^i = x_3^{-i } $ for some $ i\in \{1, 2, \dots, t-1\} .$ Then $[yx_3^i \langle x_1 \rangle]^{-1} = yx_3^i\langle x_1 \rangle.$ Suppose $  yx_3^i \langle x_1 \rangle $ contains an involution. Then there exists $ k_1 \in \{ 0, 1, \dots, 2^{l}-1 \} $ such that $yx_3^ix_1^{k_1}  =  (yx_3^ix_1^{k_1})^{-1}.$ This implies $yx_3^ix_1^{k_1}  = x_1^{-k_1}x_3^{-i}y^{-1}.$ Then it follows from equations $(\ref{eqn3})$ and $(\ref{eqn5})$ that $yx_1^{k^{i}k_1}x_3^{i}  = x_1^ {-k_1}y^{-k^{t_1-i}}x_3^{-i}.$ Consequently, $y^{1- 2k_1k^{i}}  = y^{2k_1-k^{t_1-i}}.$ Therefore, $y^{1-2k_1k^{i}-2k_1+k^{i}}  =1 .$ This implies $(2k_1-1)(k^i+1) = s'2^{l+1}$ for some $ s'\in\mathbb{Z}.$ Then $ k^i  \equiv -1 \pmod{2^{l+1}}.$ Since $ k <2^l  $ and $k$ is odd, therefore, we arrive at a contradiction due to the Lemma \ref{k2}.\\
    
\noindent \textbf{Case II:} When $ x_3^i \neq x_3^{-i } $ for $ i\in \{1, 2, \dots, t-1\} .$ Then  $  [yx_3^i \langle x_1 \rangle] $ does not contain an involution.\\
        
\noindent \textbf{Case III:} When $ x_3^i = x_3^{-is} $ for some $ i\in \{1, 2, \dots, t-1\} $. Suppose $  [yx_3^ix_2\langle x_1 \rangle] $ contains an involution. Then there exists $  m \in \{ 0, 1,\dots, 2^{l}-1 \} $ such that 
\begin{align*}
yx_3^ix_2x_1^{m} & =  (yx_3^ix_2x_1^{m})^{-1}\\
\iff yx_3^ix_2 x_1^m & = x_2 x_1^{-mk}x_3^{-i}y^{-1}\hspace{3cm} (\text{from}\; (\ref{eqn4}))\\
\iff yx_2x_3^{is}x_1^m & = x_2 x_1^{-mk}y^{-k^{t_1-i}}x_3^{-i}\hspace{2.1cm} (\text{from}\; (\ref{eqn1}),\; (\ref{eqn3}))\\
\iff x_2y^kx_3^{is}x_1^m & = x_2 x_1^{-mk}y^{-k^{t_1-i}}x_3^{-i}\hspace{2.1cm} (\text{from}\; (\ref{eqn7}))\\
\iff y^kx_3^{-i}x_1^m & = x_1^{-mk}y^{-k^{t_1-i}}x_3^{-i}\hspace{2.457cm} (\text{as} \; x_3^{is} = x_3^{-i})\\
\iff y^kx_3^{t_1-i}x_1^m & = x_1^{-mk}y^{-k^{t_1-i}}x_3^{-i}\\
\iff y^kx_1^{mk^{t_1-i}}x_3^{-i} & = x_1^{-mk}y^{-k^{t_1-i}}x_3^{-i}\hspace{2.5cm} (\text{from}\; (\ref{eqn8}))\\
\iff y^{k-2mk^{t_1-i}} & = y^{2mk-k^{t_1-i}}.\\
 \end{align*}
This shows that, $ {k-2mk^{t_1-i}-2mk+k^{t_1-i}} \equiv 0 \pmod{2^{l+1}} $. This follows that $ (2m-1)(1+k^{t_1-i-1})\equiv 0 \pmod{2^{l+1}}$ as $ k $ is odd. Consequently, $(k^{t_1-i-1}+1)\equiv 0 \pmod{2^{l+1}}$, which is a contradiction due to the Lemma \ref{k2}.\\

\noindent \textbf{Case IV:} When $ x_3^i \neq x_3^{-is} $ for some $ i\in \{1, 2, \dots, t-1\}. $ Then $ [yx_3^{i}x_2\langle x_1 \rangle]^{-1} \neq yx_3^{i}x_2\langle x_1 \rangle   $. Consequently, $ yx_3^{i}x_2\langle x_1 \rangle $ contains no involution.

\noindent Hence, it follows from equation $(\ref{eqn6})$ that $ yH_1 $ is inverse-closed and contains no involution.
\end{proof}

\begin{example}
Let $ G = S_{11}$, $ K = \langle (4\;8)(5\;6), (3\;7)(4\;8), (1\;8\;2\;4)(3\;5\;7\;6), (9\;10\; 11) \rangle $, $ H = \langle (4\;8)(5\;6), (3\;7)(4\;8), (1\;8\;2\;4)(3\;5\;7\;6) \rangle $ and $ H_1 = \langle (3\;7)(4\;8), (1\;8\;2\;4)(3\;5\;7\;6) \rangle $. Then $ H $ is a Sylow $2$-subgroup of $ K $ and let $ H_1 $ be a $ 2 $ subgroup of $ G $. It follows from Theorem \ref{thm10} that $ H $ is not perfect in $ G $. Also by Theorem \ref{thm8}, $ H_1$ is not perfect in $ G $. Therefore, $ K $ is not perfect in $ G $  due to Lemma \ref{lm1}.
\end{example}
Then Theorem \ref{thm10} together with Corollary \ref{cor9} implies the following result.
\begin{cor}
  Let $ \langle x_1, x_2 \rangle $ be a non-commutative subgroup of order $2^{l+1} $ such that $ y^{2}=x_1^{-1} $  and $o(x_1) = 2^l $, $ o(x_2) = 2$ for some $ y\in S_n$, $ l\in \mathbb{N}$. Let $ x_ix_j^{-1}x_i\in \langle x_j\rangle $ and $x_iy^{-1} = y^{-k}x_i $ some $ m, k\in\mathbb{N}$ and for $ i,j\in\{2,3,\dots, m\}$. If $ k $ is an odd integer and $k < 2^l $, then $\langle x_1, x_2,\dots, x_m \rangle $ is not perfect in $S_n$. 
\end{cor}

\section{Conclusion}
This study provides a comprehensive classification of cyclic $2$-subgroup perfect codes in $ S_n$. Through detailed analysis, we have elucidated the structure and properties of these subgroup codes. Furthermore, our findings extend to a broader range of subgroup codes within the symmetric group $ S_n $, addressing both commutative and non-commutative instances. Moreover, these results help to examine subgroup perfect codes in any finite group if we classify square elements of that group.\\

%%%%%%%%%%%%%%%%%%%%%%%%%%%%%%%%%%%%%

\noindent
\textbf{Author contributions} Ankan Shaw: Conceptualization, Resources, Investigation, Writing -- original draft. Shibesh Kotal: Investigation, Calculation, Writing. Satya Bagchi: Supervision, Review \& editing.\\

\noindent
\textbf{Funding} The first author of the paper would like to thank CSIR-HRDG, India, for support financially to carry out this work. This
work is also supported by the Science \& Engineering Research Board
(SERB), Government of India (Grant No. MTR/2021/000611).\\

\noindent
\textbf{Declarations}

\noindent
\textbf{Ethics approval} This article contains no human subjects, and informed consent is not required.\\

\noindent
\textbf{Consent to publish} All the authors involved in this manuscript give full consent for publication of this submitted article.\\

\noindent
\textbf{Conflict of interest} The authors declare that they have no conflict of interest.

%%%%%%%%%%%%%%%%%%%%%%%%%%%%%%%%%%%%%%%%%%%%%

\end{document}